\documentclass[10pt]{amsart}

\usepackage[english]{babel}

\newtheorem{theorem}{Theorem}[section]
\newtheorem{lemma}[theorem]{Lemma}
\newtheorem{proposition}[theorem]{Proposition}
\newtheorem{corollary}[theorem]{Corollary}

\theoremstyle{definition}
\newtheorem{definition}[theorem]{Definition}

\theoremstyle{remark}
\newtheorem{remark}[theorem]{Remark}

\numberwithin{equation}{section}
\DeclareMathOperator*{\fil}{D}
\newcommand{\N}{\mathbb N}
\newcommand{\Z}{\mathbb Z}
\newcommand{\R}{\mathbb R}
\newcommand{\T}{\mathfrak{T}}
\newcommand{\e}{\varepsilon}
\let\phi=\varphi
\DeclareMathOperator*{\conv}{conv}
\DeclareMathOperator*{\diam}{diam}
\DeclareMathOperator*{\dist}{dist}

\begin{document}

\setcounter{page}{1}

\title{Homologically maximizing geodesics in conformally flat tori}

\author[Suhr]{Stefan Suhr}
\address{Mathematisches Institut, Universit\"at Freiburg}
\email{suhr100@googlemail.com}

\date{\today}

\begin{abstract}
We study homologically maximizing timelike geodesics in conformally flat tori. A causal geodesic 
$\gamma$ in such a torus is said to be homologically maximizing if one (hence every) lift of $\gamma$ to the 
universal cover is arclength maximizing. First we prove a compactness result for homologically maximizing timelike 
geodesics. This yields the Lip\-schitz continuity of the time separation of the universal cover on strict sub-cones 
of the cone of future pointing vectors. Then we introduce the stable time separation $\mathfrak{l}$. As an 
application we prove relations between the concavity properties of $\mathfrak{l}$ and the qualitative behavior of 
homologically maximizing geodesics.
\end{abstract}
\maketitle

\section{Introduction}

Here, we present a version of Mather theory for maximizing geodesics on conformally flat Lorentzian tori. More 
general Lorentzian manifolds will be treated in \cite{diss}. The source for the techniques we employ are \cite{bu} 
and \cite{ba}.

Consider a real vector space $V$ of dimension $m<\infty$ and $\langle .,.\rangle_1$ a nondegenerate symmetric 
bilinear form on $V$ with signature $(-,+,\ldots ,+)$. Set $|.|_1:=\sqrt{|\langle .,.\rangle_1|}$. Further let 
$\Gamma\subseteq V$ be a co-compact lattice and $f\colon V \to (0,\infty)$ a smooth $\Gamma$-invariant function. 
The Lorentzian metric $\overline{g}:= f^2\langle .,.\rangle_1$ then descends to a Lorentzian metric on the torus 
$V/\Gamma$. Denote the induced Lorentzian metric by $g$. Choose a time-orientation of $(V,\langle .,.\rangle_1)$. 
This time-orientation induces a time-orientation on  $(V/\Gamma,g)$ as well. Note that $(V/\Gamma,g)$ is vicious 
(\cite{bee} p. 137) and the universal cover $(V,\overline{g})$ is globally hyperbolic (\cite{bee} p. 65). According 
to \cite{rosa1} proposition 2.1, $(V/\Gamma,g)$ is geodesically complete in all three causal senses. Fix a norm 
$\|.\|$ on $V$ and denote the dual norm by $\|.\|^\ast$. We define $B_r(x):=\{y\in V|\;\|y-x\|<r\}$. Note that 
$\|.\|$ induces a metric on $V/\Gamma$. For a subset 
$A\subseteq V$ we write $\dist(x,A)$ to denote the distance of the point $x\in V$ to $A$ relative to $\|.\|$. 
Further denote by $\T$ the positive oriented causal vectors of $(V,\langle .,.\rangle_1)$, i.e. the vectors $v\in 
V\setminus \{0\}$ with $\langle v,v\rangle_1\le 0$ and positive time-oriented. For $\e>0$ set 
$\T_\e:= \{v\in \T|\, \dist(v,\partial \T)\ge \e\|v\|\}$.

Let $I$ be any (bounded or unbounded) interval in the reals.
A causal geodesics $\gamma \colon I\to V/\Gamma$ of $(V/\Gamma, g)$ is said to be homologically maximizing if one 
(hence every) lift $\overline{\gamma} \colon I\to V$ is arclength maximizing in $(V,\overline{g})$ (for simplicity 
we will only consider future pointing curves) in the following sense: For every compact subinterval $[a,b]\subseteq 
I$ the curve $\gamma|_{[a,b]}$ is arclength maximizing among all causal curves connecting $\gamma(a)$ to 
$\gamma(b)$. In section \ref{S2} we will prove a compactness
result for homologically maximizing timelike geodesics. Using this compactness result we will then deduce
the Lipschitz continuity of the time separation of $(V,\overline{g})$ on $\{(x,y)\in V\times V|\, y-x \in\T_\e\}$ 
for every $\e>0$. Here we use the term time separation as a synonym for the Lorentzian distance function (\cite{bee} 
definition 4.1).

In section \ref{S4} we show the existence of the stable version of the time separation $d$ of $(V,\overline{g})$, 
i.e. $\mathfrak{l}(v)=\lim_{n\to \infty}\frac{d(x,x+nv)}{n}$ exists for all $x\in V$ and $v\in \T$ and is independent 
of $x$. We will call $\mathfrak{l}$ the {\it stable time separation} of $(V/\Gamma,g)$. 
Futhermore for any $\e>0$ there exists a constant $K(\e)<\infty$ such that $|d(x,x+v)-\mathfrak{l}(v)|\le K(\e)$ 
for all $x\in V$ and $v\in \T_\e$. The stable time separation constitutes the Lorentzian version of the stable norm 
on $H_1(M,\R)$ of a compact Riemannian manifold $(M,g_R)$. The strategy of deduction we follow is taken from 
\cite{bu}. Even in the Riemannian case, the mentioned estimate on $|d(x,x+v)-\mathfrak{l}(v)|$ is not obvious.

In section \ref{S5} we relate concavity properties of $\mathfrak{l}$ to the existence and the asymp\-totic properties 
of homologically maximizing geodesics. More precisely, we show that for any homologically maximizing geodesic 
$\gamma\colon \R\to V/\Gamma$ there 
exists a support function $\alpha$ of $\mathfrak{l}$ such that all accumulation points of sequences of rotation 
vectors of subarcs of $\gamma$ lie in the intersection $\alpha^{-1}(1)\cap \mathfrak{l}^{-1}(1)$. Conversely, for any 
support function $\alpha$ of $\mathfrak{l}$ we can find a homologically maximizing timelike geodesic such that the 
limits of rotation vectors of $\gamma$ lie in $\alpha^{-1}(1)\cap \mathfrak{l}^{-1}(1)$. As a corollary we obtain 
the existence of infinitely many geometrically distinct homologically maximizing timelike geodesics in $(V/\Gamma,g)$. 

{\it Acknowledgement:} I would like to thank Prof. V. Bangert for the excellent support in the preparation of my 
diploma thesis, out of which these notes have arisen.

\section{Compactness Theorems}\label{S2}

For a curve $\gamma\colon I\to V/\Gamma$ and $s,t\in I$ set $\gamma(t)-\gamma(s):=\overline{\gamma}(t)-
\overline{\gamma}(s)$, where $\overline{\gamma}\colon I\to V$ is any lift of $\gamma$. Obviously, this definition 
does not depend on the chosen lift $\overline{\gamma}$.

\begin{definition}\label{D1}
(i) Let $\e >0$ and $G<\infty$.
A causal curve $\gamma\colon I\to V/\Gamma$ is said to be $(G,\e)$-timelike if there exist $a,b\in I$ such that 
$\gamma(b)-\gamma(a)\in \T_\e$ and $\|\gamma(b)-\gamma(a)\|\ge G$.

(ii) Let $F<\infty$. A causal curve $\gamma\colon I\to V/\Gamma$ is said to be $F$-almost maximal if 
$$L^{g}(\gamma|_{[s,t]})\ge d(\overline{\gamma}(s),\overline{\gamma}(t))-F$$ 
for one (hence every) lift $\overline{\gamma}$ of $\gamma$ to $V$ and all $[s,t]\subseteq I$.
\end{definition}

\begin{proposition}\label{P1}
For every $\e>0$ and $F<\infty$ there exist constants $\delta>0$ and $0<K<\infty$ such that for all $G<\infty$, 
all $F$-almost maximal $(G,\e)$-timelike curves $\gamma\colon I \to V/\Gamma$ and all $s< t\in I$ with 
$\|\gamma(t)-\gamma(s)\|\ge K$ we have
$$\gamma(t)-\gamma(s)\in \T_\delta.$$
\end{proposition}

Before giving the proof of proposition \ref{P1} we review some applications.

Choose a orthonormal basis $\{e_1,\ldots ,e_m\}$ of $(V,\langle .,.\rangle_1)$. Note that the translations $x\mapsto 
x+v$ are conformal diffeomorphisms of $(V,\overline{g})$ for all $v\in V$. Then the $\overline{g}$-orthogonal frame 
field $x\mapsto (x,(e_1,\ldots ,e_m))$ on $V$ descends to a $g$-orthogonal frame field on $V/\Gamma$. 
In this way it makes sense to speak of a tangent vector $w\in T(V/\Gamma)$ as belonging to $\T$ or
$\T_\e$ for $\e>0$.
\begin{theorem}\label{T2}
For every $\e>0$ there exists a $\delta>0$ such that for all future pointing homologically maximizing geodesics 
$\gamma\colon I \to V/\Gamma$ with $\dot\gamma(t_0)\in \T_\e$ for some $t_0\in I$, we have
$$\dot\gamma(t)\in \T_\delta$$
for all $t\in I$.
\end{theorem}
Theorem \ref{T2} is a direct consequence of proposition \ref{P1} together with the continuity of the geodesic 
flow and the invariance of the set of lightlike vectors under the geodesic flow. 
This has the following immediate consequence. 
\begin{corollary}\label{C1}
Let $\e>0$ and $G<\infty$. Then any limit curve of a sequence of homologically maximizing $(G,\e)$-timelike curves 
in $(V/\Gamma,g)$ is timelike. 
\end{corollary}
Note that limit curves are understood in the sense of the limit curve lemma (\cite{bee} lemma 14.2).
The corollary resembles the {\it generalized timelike co-ray condition} in \cite{gaho}. It requires that any co-ray 
to a given timelike ray is again timelike. Following \cite{e1} it was proved in 
\cite{gaho} that the generalized timelike co-ray condition implies the Lipschitz continuity of the Busemann function
associated to a given timelike ray. The same proof (with some obvious modifications) works in the present situation 
as well.

\begin{theorem}\label{T2a}
For all $\e>0$ there exists an $L=L(\e)<\infty$ such that the time separation $d$ of $(V,\overline{g})$ is 
$L$-Lipschitz on $\{(x,y)\in V\times V|\, y-x\in \T_\e\}$.
\end{theorem}

Now we proceed to the proof of proposition \ref{P1}.

First note the following fact. There exist constants $c>0$ and $C<\infty$ such that 
\begin{equation}\label{E1}
c\|w\|\dist(v,\partial\T)\le |\langle v,w\rangle_1|   \text{ and }|v|_1^2\le C \|v\|\dist(v,\partial\T)
\end{equation}
for all $v,w\in\T$. 

Since $\langle .,.\rangle_1$ is non-degenerate there exist constants $c'>0$ and $C'<\infty$ with 
$$c'\|v\|\le \|\langle v,.\rangle_1\|^\ast\le C'\|v\|$$
for all $v\in V$. For the first inequality in (\ref{E1}) note that for every 
$w\in\T$ the orthogonal complement of $w$ relative to $\langle .,.\rangle_1$, denoted by $w^{\perp}$, 
is a spacelike hyperplane. Since $\langle .,.\rangle_1$ is non-degenerate, we have $\dist(v,\partial \T)\le 
\dist(v,w^{\perp})$ for every $v\in \T$. Consider for $v\in V$ a $v_0\in w^{\perp}$ with $\|v-v_0\|=
\dist(v,w^{\perp})$. Note that we have 
$$|\langle v-v_0,w\rangle_1|=\|v-v_0\|\;\|\langle .,w\rangle_1\|^\ast.$$
Consequently we get
$$|\langle v,w\rangle_1| =|\langle v_0-v,w\rangle_1|\ge c'\|w\|\,\|v-v_0\|=c' \|w\| \dist(v,w^{\perp}).$$
Next consider a $v_1\in\partial\T$ with $\|v-v_1\|=\dist(v,\partial\T)$. Note that $\|v_1\|\le 2\|v\|$. Then we have
\begin{align*}
-\langle v,v\rangle_1&=-\langle v,v\rangle_1+\langle v_1,v_1\rangle_1=-2\int_0^1\langle (1-t)v_1+tv,v-v_1\rangle_1 
dt\\
&\le 2C'\sup_{t\in[0,1]}\|(1-t)v_1+tv\| \|v-v_1\|\le 4C'\|v\|\,\|v-v_1\|\\
&=:C\|v\|\dist(v,\partial\T).
\end{align*}

Since $\T$ contains no linear subspaces we can choose $\eta>0$ such that 
\begin{equation}\label{E1c}
\|\sum v_i\|\ge \eta\sum \|v_i\|
\end{equation} 
for any finite set $\{v_i\}_{1\le i\le N}\subset \T$. Note that this implies that we have
\begin{equation}\label{E1d}
L^{\|.\|}(\gamma)\ge \eta \|\gamma(b)-\gamma(a)\|
\end{equation} 
for any future pointing curve $\gamma\colon [a,b]\to V/\Gamma$.

\begin{lemma}\label{L1}
Let $\e,\lambda_0 >0$. Then there exists $\e_0>0$ such that for all $\{v_i\}_{1\le i\le N}\subseteq \T$ 
with $\sum v_i\in \T_\e$ and $\sum |v_i|_1 \ge \lambda_0 |\sum v_i|_1$ there exists $j\in\{1,\ldots ,N\}$ with 
$v_j\in \T_{\e_0}$.
\end{lemma}

\begin{proof}
Assume that for every $n\in\Z_{>0}$ there exist $\{v_i^n\}_{1\le i\le N(n)}\subseteq \T$ with $\sum v_i^n\in \T_\e$ 
and $v_i^n \notin \T_{\frac{1}{n}}$. With (\ref{E1}) and (\ref{E1c}) we have
\begin{align*}
0<\lambda_0\sqrt{c\e}\left\|\sum v_i\right\|&\le \lambda_0\left|\sum v_i^n\right|_1\\
&\le \sum |v_i^n|_1\le \sqrt{\frac{C}{n}}\sum \|v_i^n\|\le \frac{1}{\eta}\sqrt{\frac{C}{n}}\left\|\sum v_i\right\|.
\end{align*}
Consequently we get $0<\lambda_0\sqrt{c\e}\le \frac{1}{\eta}\sqrt{\frac{C}{n}}\to 0$ for $n\to \infty$.
\end{proof}

\begin{lemma}\label{L1a}
$v\in \T\mapsto \dist(v,\partial\T)$ is a concave function.
\end{lemma}

\begin{proof}
We prove that the superlevels $\dist^{-1}(.,\partial\T)([r,\infty))$ are convex for all $r\in\R$. 

Fix $r\in \R$ and let $v,w\in \dist^{-1}(.,\partial\T)([r,\infty))$. For $\lambda \in [0,1]$ set 
$A_\lambda:=B_r((1-\lambda)v+\lambda w)$. We claim that 
\begin{equation}\label{E1a}
\cup_{\lambda\in [0,1]}A_\lambda =\conv(A_0\cup A_1),
\end{equation}
where $\conv(A_0\cup A_1)$ denotes the convex hull of $A_0\cup A_1$.
Let $x\in \conv(A_0\cup A_1)$. With the theorem of Caratheodory we can choose $x_0,\ldots ,x_m\in A_0\cup A_1$ 
and $\lambda^0,\ldots ,\lambda^m\ge 0$ with $\sum \lambda^i=1$ and $x= \sum\lambda^i x_i$. By relabeling the
$x_i$ we can assume that $x_0,\ldots ,x_j\in A_0$ and $x_{j+1},\ldots ,x_m\in A_1$. Then we have 
$$x_v:=\sum_{k=0}^j \lambda^k (x_k-v), x_w:=\sum_{l=j+1}^m \lambda^l (x_l-w)\in B_r(0).$$
Set $\lambda :=\sum_{k=0}^j \lambda^k$. We get $x_v+x_w=x-\lambda v-(1-\lambda)w$. Since $\|x_v+x_w\|\le r$ we have
$x\in A_\lambda$. Therefore we get $\conv(A_0\cup A_1)\subseteq \cup A_\lambda$. 

For the other inclusion let $y\in \cup A_\lambda$. Choose $\lambda_y\in [0,1]$ with $y\in A_{\lambda_y}$. Define 
\begin{align*}
y_v:=y+\lambda_y (v-w) \text{ and } y_w:=y+(1-\lambda_y)(w-v).
\end{align*}
We have $y_v\in A_0$, $y_w\in A_1$ and $y\in \conv(\{y_v,y_w\})$. This shows (\ref{E1a}).

Since 
$$\cup A_\lambda = \conv(A_0\cup A_1)\subset \conv(\T\cup\T)=\T$$ 
we have $\dist(\lambda v+(1-\lambda)w,\partial \T)\ge r$ for all $\lambda \in [0,1]$. This implies the 
convexity of $\dist^{-1}(.,\partial\T)([r,\infty))$.
\end{proof}

Note that $v\in \T\mapsto \dist(v,\partial\T)$ is positively homogenous of degree one, i.e. we have 
$\dist(\lambda v,\partial \T)=\lambda \dist(v,\partial \T)$ for all $\lambda \ge 0$. 

Lemma \ref{L1a} and the positive homogeneity imply
\begin{equation}\label{E1b}
\dist(v+w,\partial \T)\ge \dist(v,\partial\T)+\dist(w,\partial\T).
\end{equation}

\begin{corollary}\label{C0}
The cones $\T_\e$ are convex for all $\e>0$. 
\end{corollary}

\begin{proof}
For $v,w\in \T_\e$ we have
\begin{align*}
\dist(v+w,\partial\T)&\ge \dist(v,\partial\T)+\dist(w,\partial \T)\ge \e(\|v\|+\|w\|)\\
&\ge \e\|v+w\|.
\end{align*}
\end{proof}

\begin{lemma}\label{L2}
Let $\lambda_1 <\infty$ and set $\mu:=\frac{c\;\eta}{4\lambda_1^2 C}$. Then we have
$$|v+w|_1 \ge \lambda_1 (|v|_1+|w|_1)$$
for all $v,w\in \T$ with $\|v\|\le \mu\|w\|$ and $\dist(w,\partial\T)\le \mu\dist(v,\partial\T)$.
\end{lemma}

\begin{proof}
Using (\ref{E1}), (\ref{E1b}) and (\ref{E1c}) we get
\begin{align*}
|v+w|_1&\ge \sqrt{c\dist(v+w,\partial\T)\|v+w\|}\ge \sqrt{c\;\eta \dist(v,\partial\T)\|w\|}\\
&=\frac{\sqrt{c\;\eta}}{2}\left(\sqrt{\dist(v,\partial\T)\|w\|}+\sqrt{\dist(v,\partial\T)\|w\|}\right)\\
&\ge \lambda_1\sqrt{C}(\sqrt{c\dist(v,\partial\T)\|v\|}+\sqrt{c\dist(w,\partial\T)\|w\|})\ge \lambda_1 (|v|_1+|w|_1).
\end{align*}
\end{proof}

Note that the time separation $d$ of $(V,\overline{g})$ satisfies
$$\inf\, f\,|v|_1\le d(x,x+v)\le \sup \, f\,|v|_1$$
for all $v\in \T$ and $x\in V$.

\begin{lemma}\label{L5}
Let $\kappa',F'<\infty$ and $\e'>0$ be given. Then there exists $K':=K'(\kappa',F',\e')<\infty$ such that 
for all $G'<\infty$ and all future pointing $(G',\e')$-timelike $F'$-almost maximal curves 
$\gamma\colon [s,t]\to V/\Gamma$ with $\|\gamma(t)-\gamma(s)\|\ge K'$, we have 
$$\dist(\gamma(t)-\gamma(s),\partial \T)\ge \kappa'.$$
\end{lemma}

\begin{proof}
Assume that the claim is false. Then there exist $\kappa',F'<\infty$, $\e'>0$, a sequence $G'_n<\infty$ and 
a sequence of $(G'_n,\e')$-timelike and $F'$-almost maximal curves $\gamma_n\colon [s_n,t_n]\to V/\Gamma$ with 
$$\|\gamma_n(t_n)-\gamma_n(s_n)\|\ge n\text{ and }\dist(\gamma_n(t_n)-\gamma_n(s_n),\partial \T)\le \kappa'.$$
Choose $[a_n,b_n]\subseteq [s_n,t_n]$ with $\|\gamma(b_n)-\gamma(a_n)\|\ge G'_n$ and 
$\gamma(b_n)-\gamma(a_n)\in \T_{\e'}$. If $\e' \|\gamma(b_n)-\gamma(a_n)\| > \kappa'$ the contradiction is obvious.  
We have 
$$\dist(\gamma(b_n)-\gamma(a_n),\partial\T)\ge \e'\|\gamma(b_n)-\gamma(a_n)\|> \kappa'.$$
Since $\dist(\gamma(t_n)-\gamma(s_n),\partial\T)\ge \dist(\gamma(b_n)-\gamma(a_n),\partial\T)$, 
we get $\dist(\gamma(t_n)-\gamma(s_n),\partial\T)>\kappa'$.

Therefore we can assume that $\e' G'_n\le \e' \|\gamma(b_n)-\gamma(a_n)\| \le\kappa'$. Note that we can further 
assume that $G'_n= \e'$, since we have $\|\gamma(b_n)-\gamma(a_n)\|\ge \dist(\gamma(b_n)-\gamma(a_n),\partial \T)$.

First we consider the case that $\|\gamma_n(t_n)-\gamma_n(a_n)\|$ is unbounded. We can pass to a subsequence and 
assume that $\|\gamma_n(t_n)-\gamma_n(a_n)\|\to \infty$. We can assume that $\gamma_n$ is parameterized by 
$\|.\|$-arclength. With (\ref{E1d}) we know that $L^{\|.\|}(\gamma_n|_{[a_n,t_n]})
=t_n-a_n\to \infty$ for $n\to\infty$. By shifting the parameter we can assume that $a_n\equiv 0$. 

According to the limit-curve lemma there exists a subsequence of $\{\gamma_n|_{[0,t_n]}\}_{n\in \N}$ converging
 uniformly on compact
sets to a future pointing curve $\gamma_\infty \colon [0,\infty)\to V/\Gamma$. It is classical that $\gamma_\infty$ 
is $F'$-almost maximal as well (\cite{bee} proposition 14.3). The fact that there exists a $T_0\ge 0$ with 
$$\|\gamma_\infty (T_0)-\gamma_\infty(0)\|\in \left[\e',\frac{\kappa'}{\e'}\right]\text{ and }
\gamma_\infty(T_0)-\gamma_\infty(0)\in \T_{\e'}$$  
is ensured by the uniform convergence and the continuity of the functions $v\mapsto \|v\|$ and $v\mapsto
\dist(v,\partial\T)$. With the same argument we get that 
$$\dist(\gamma_\infty(T)-\gamma_\infty(0),\partial\T)\le \kappa'$$
for all $T\ge 0$. 

Set $\lambda_1 :=\frac{2\sup f}{\inf f}$ and $\mu:=\frac{c\;\eta}{4\lambda_1^2 C}$. Choose $T_1>T_0$ such that we 
have 
$$\dist(\gamma_\infty(T)-\gamma_\infty(T_1),\partial\T)\le \mu\; \e'^2
\le \mu \dist(\gamma_\infty(T_1)-\gamma_\infty(0),\partial\T)$$
for all $T>T_1$. Further choose $T_2>T_1$ with 
$$\|\gamma_\infty(T_2)-\gamma_\infty(T_1)\|\ge \mu \|\gamma_\infty(T_1)-\gamma_\infty(0)\|$$
and 
\begin{equation}\label{E3}
\|\gamma_\infty(T_2)-\gamma_\infty(0)\|\ge \frac{4F'^2}{c\,\e'^2\inf f^2}+1.
\end{equation}
On the one hand we get 
$$\frac{\inf f}{2} |\gamma_\infty(T_2)-\gamma_\infty(0)|_1\ge \sup f (|\gamma_\infty(T_2)-\gamma_\infty(T_1)|_1
+|\gamma_\infty(T_1)-\gamma_\infty(0)|_1)$$
with lemma \ref{L2}. On the other hand, using (\ref{E3}), we have  
\begin{align*}
\inf f|\gamma_\infty(T_2)-\gamma_\infty(0)|_1&\ge \inf f\sqrt{c\dist(\gamma_\infty(T_2)-\gamma_\infty(0),\partial\T)
\|\gamma_\infty(T_2)-\gamma_\infty(0)\|}\\
&\ge \inf f\sqrt{c\dist(\gamma_\infty(T_0)-\gamma_\infty(0),\partial\T)\|\gamma_\infty(T_2)-\gamma_\infty(0)\|}\\
&\ge \inf f\;\e'\sqrt{c\;\|\gamma_\infty(T_2)-\gamma_\infty(0)\|} > 2F'.
\end{align*}
Therefore we get
\begin{align*}
L^g(\gamma_\infty|_{[0,T_2]})&\ge \inf f |\gamma_\infty(T_2)-\gamma_\infty(0)|_1-F'\\
&>\frac{\inf f}{2} |\gamma_\infty(T_2)-\gamma_\infty(0)|_1\\
&\ge \sup f (|\gamma_\infty(T_2)-\gamma_\infty(T_1)|_1+|\gamma_\infty(T_1)-\gamma_\infty(0)|_1)\\
&\ge L^g(\gamma_\infty|_{[0,T_1]})+L^g(\gamma_\infty|_{[T_1,T_2]}).
\end{align*}
Consequently the sequence $\|\gamma_n(t_n)-\gamma_n(a_n)\|$ has to be bounded.

In the other case $\|\gamma_n(b_n)-\gamma_n(s_n)\|\to \infty$, we obtain an analogous contradiction.
Since 
$$\|\gamma_n(t_n)-\gamma_n(s_n)\|\le \|\gamma_n(b_n)-\gamma_n(s_n)\|+\|\gamma_n(t_n)-\gamma_n(a_n)\|,$$
we have a contradiction to the assumption that $\|\gamma_n(t_n)-\gamma_n(s_n)\|\to \infty$ for $n\to\infty$.
This finishes the proof.
\end{proof}

Denote by 
$$\diam(\Gamma,\|.\|):=\frac{1}{2}\inf\left\{\max_{1\le i\le m}\|k_i\||\;\langle k_1,\ldots ,k_m\rangle_\Z=\Gamma\right\},$$
where $\langle k_1,\ldots ,k_m\rangle_\Z$ denotes the $\Z$-span of $k_1,\ldots ,k_m\in \Gamma$.
Since $\Gamma$ is a co-compact lattice there exists for all $x,y\in V$ a $l_{x,y}\in\Gamma$ with 
$$\|x-(y+l_{x,y})\|\le \diam(\Gamma,\|.\|).$$

\begin{lemma}
There exists $\fil<\infty$ such that for all $x,y\in V$ there exists $k_{x,y}\in \Gamma$ with 
$\|x-(y+k_{x,y})\|\le \fil$ and $y+k_{x,y}\in x+\T$.
\end{lemma}

\begin{proof}
Choose $v\in \T\setminus \partial\T$ and $\e>0$ such that we have $B_\e(v)\subseteq \T$. Since $\T$ is a cone we have
$B_{\lambda \e}(\lambda v)\subseteq \T$ for all $\lambda \ge 0$. Choose $\frac{\diam(\Gamma,\|.\|)}{\e}
<\Lambda <\infty$. Then we have $B_{\diam(\Gamma,\|.\|)}(\Lambda v)\subseteq \T$. 
Set $k_{x,y}:=l_{x+\Lambda v,y}$ and $\fil:= \diam(\Gamma,\|.\|) +\Lambda \|v\|$.
\end{proof}

\begin{proof}[Proof of proposition \ref{P1}]
Let $F,G<\infty$, $\e>0$ and $\gamma\colon I\to V/\Gamma$ be a $F$-almost maximal $(G,\e)$-timelike curve.

(i) Choose $[a',b']\subseteq I$ with $\|\gamma(b')-\gamma(a')\|\ge G$ and $\gamma(b')-\gamma(a')\in\T_\e$. 
Set $G_0:=\max\left\{\frac{2F}{\sqrt{c\,\e}\inf f},\e\right\}$. If $G\ge G_0$ we get
\begin{align*}
\inf f|\gamma(b')-\gamma(a')|_1&\ge \inf f\sqrt{c\dist(\gamma(b')-\gamma(a'),\partial\T)\|\gamma(b')-\gamma(a')\|}\\
&\ge \inf f\sqrt{c\;\e}\;\|\gamma(b')-\gamma(a')\|\ge \inf f \sqrt{c\;\e}\; G\ge 2F.
\end{align*}
For any partition $\{\gamma|_{[a_i,b_i]}\}_{1\le i\le N}$ of $\gamma|_{[a',b']}$ we have
\begin{align*}
\sup\,f \sum |\gamma(b_i)-\gamma(a_i)|_1&\ge \sum L^g(\gamma|_{[a_i,b_i]})=L^g(\gamma|_{[a',b']})\\
&\ge \inf\,f |\gamma(b')-\gamma(a')|_1-F\\
&\ge \frac{\inf\,f}{2}|\gamma(b')-\gamma(a')|_1.
\end{align*}
Apply lemma \ref{L1} to $v_i:=\gamma(b_i)-\gamma(a_i)$ with $\lambda_0 :=\frac{\inf\,f}{2\sup\,f}$ and $\e>0$ 
as above. Consequently there exist $\e_0>0$ and $[a,b]\subseteq [a',b']$ with 
$\|\gamma(b)-\gamma(a)\|\in [\e,2G_0]$ (note that $\dist(v,\partial\T)\le \|v\|$ for all $v\in\T$) 
and $\gamma(b)-\gamma(a)\in \T_{\e_0}$. 

(ii) Let $s<a<b<t \in I$. If $\|\gamma(b)-\gamma(s)\|\ge \nu:=\frac{4(F+1)^2}{\inf f^2 c\;\e_0 \e}$ we get
\begin{align*}
\inf f|\gamma(b)-\gamma(s)|_1&\ge \inf f\sqrt{c\dist(\gamma(b)-\gamma(s),\partial\T)\|\gamma(b)-\gamma(s)\|}\\
&\ge \inf f\sqrt{c\;\e_0 \e\|\gamma(b)-\gamma(s)\|}\ge 2(F+1).
\end{align*}
Consequently we have 
$$\sup f(|\gamma(a)-\gamma(s)|_1+|\gamma(b)-\gamma(a)|_1)> \frac{\inf f}{2} |\gamma(b)-\gamma(s)|_1.$$
Set $\lambda_1:=\frac{2\sup f}{\inf f}$. Recall the definition of $\mu:=\frac{c\;\eta}{4\lambda_1^2 C}$. If 
$\|\gamma(a)-\gamma(s)\|\ge \mu\; (2 G_0)$ we get 
$$\dist(\gamma(a)-\gamma(s),\partial \T)> \mu\dist(\gamma(b)-\gamma(a),\partial \T)\ge \mu\; \e_0 \e$$
with lemma \ref{L2}. Note that we have 
$$\|\gamma(a)-\gamma(s)\|\ge \|\gamma(b)-\gamma(s)\|-\|\gamma(b)-\gamma(a)\|\ge \nu-2 G_0.$$
Consequently we get that if 
$$\sup_{s'\in I,\; s'<a}\|\gamma(b)-\gamma(s')\|\ge \max\left\{\nu-2 G_0,2 \mu\; G_0\right\}:=H_0$$
there exists $s\in I$, $s<a$ with 
$$\|\gamma(a)-\gamma(s)\|= H_0\text{ and }\gamma(a)-\gamma(s)\in \T_{\delta_0}$$
for $\delta_0:=\frac{\mu\;\e_0\e}{H_0}$.
In the same way we obtain the existence of a parameter $t\in I$, $t>b$ with 
$\|\gamma(t)-\gamma(b)\|= H_0$ and $\gamma(t)-\gamma(b)\in \T_{\delta_0}$,
if we have that $\sup_{t'\in I,\;t'>b}\|\gamma(t')-\gamma(b)\|\ge H_0$.

(iii) Define $\kappa':=3\fil$, $G':=2\mu\; G_0$, $F':=F$ and $\e':=\delta_0$ and $K_0:=K'(3\fil, 2\mu G_0,F,\delta_0)$, 
according to lemma \ref{L5}. Then there exists $\tau\in I$, $b<\tau$ with 
$$\|\gamma(\tau)-\gamma(b)\|=K_0\text{ and }\dist(\gamma(\tau)-\gamma(b),\partial\T)\ge 3\fil$$
if $\sup_{t'\in I,\;t'>b}\|\gamma(t')-\gamma(b)\|\ge K_0$.
Analogously, if we have $\sup_{s'\in I,\;s'<a}\|\gamma(a)-\gamma(s')\|\ge K_0$, there exists $\sigma\in I$, $\sigma<a$ with 
$$\|\gamma(a)-\gamma(\sigma)\|=K_0\text{ and }\dist(\gamma(a)-\gamma(\sigma),\partial\T)\ge 3\fil.$$
We saw in step (ii) that for intervals $[b,t]\subseteq I$ with $\|\gamma(t)-\gamma(b)\|=H_0$ we have 
$$\gamma(t)-\gamma(b)\in \T_{\delta_0}.$$
We will carry this over to all intervals $[s,t]\subset I$ with $\|\gamma(t)-\gamma(s)\|$ sufficiently large via the 
following cut-and-paste argument.

Note first that it suffices to consider the case $\sup_{t'\in I,\;t'>b}\|\gamma(t')-\gamma(b)\|\ge K_0$.
The case $\sup_{s'\in I,\;s'<a}\|\gamma(a)-\gamma(s')\|\ge K_0$ can be reduced to the former by considering 
$\gamma_{\text{inv}}(t):=\gamma(-t)$ and the opposite time-orientation on $(V/\Gamma,g)$. 

Therefore we can assume that $\sup_{t'\in I,\,t'>b}\|\gamma(t')-\gamma(b)\|\ge K_0$. Choose $\tau\in I$, $b<\tau$ 
with $\|\gamma(\tau)-\gamma(b)\|=K_0$. We have $\dist(\gamma(\tau)-\gamma(b),\partial\T)\ge 3\fil$. 
Let $[s,t]\subset I$ be an interval mutually disjoint to $[a,\tau]$.
We can choose future pointing curves $\zeta_i\colon [a_i,b_i]\to V/\Gamma$ $(i=1,\ldots ,6)$ with 
$L^{\|.\|}(\zeta_{1,2,4,5})\le \fil$ such that 
$$\gamma':=\gamma|_{[a,b]}\ast \zeta_1\ast\gamma|_{[s,t]}\ast\zeta_2\ast \gamma|_{[\tau,s]}\ast\zeta_3$$
is future pointing and homotopic with fixed endpoints to $\gamma|_{[a,t]}$ if $s\ge \tau$ and such that 
$$\gamma'':=\gamma|_{[s,t]}\ast\zeta_4\ast\gamma|_{[a,b]}\ast\zeta_5\ast\gamma|_{[t,a]}\ast\zeta_6$$
is future pointing and homotopic with fixed endpoints to $\gamma|_{[s,\tau]}$ if $t\le a$.

This can be seen as follows: Assume first that $s\ge \tau$.
Choose future pointing curves $\zeta_1,\zeta_2$ with $L^{\|.\|}(\zeta_i)\le \fil$ 
connecting $\gamma(b)$ with $\gamma(s)$ and $\gamma(t)$ with $\gamma(\tau)$.
Now consider a lift $\overline{\gamma}$ of $\gamma$ to $V$ and a lift $\overline{\xi}_1$ of 
$\zeta_1\ast\gamma|_{[s,t]}\ast\zeta_2\ast \gamma|_{[\tau,s]}$ starting at
$\overline{\gamma}(b)$. Let $q$ be the terminal point of $\overline{\xi}_1$. Then we have
$$\overline{\gamma}(t)-q=[\gamma(\tau)-\gamma(t)]+\sum_{i=1}^2 [\zeta_i(b_i)-\zeta_i(a_i)].$$
By construction we have $\|\sum_{i=1}^2 [\zeta_i(b_i)-\zeta_i(a_i)]\|\le 2\fil$. Since 
$\dist(\gamma(\tau)-\gamma(t),\partial\T)\ge 3\fil$ we get 
$\overline{\gamma}(t)-q\in \T$. Choose a future pointing curve $\zeta_3\colon [a_3,b_3]\to V/\Gamma$ with 
$\zeta_3(a_3)=\gamma(s)$ and $\zeta_3(b_3)-\zeta_3(a_3)=\overline{\gamma}(t)-q$. This completes the construction 
of $\gamma'$. 

If $t\le a$ choose future pointing curves $\zeta_4,\zeta_5$ with $L^{\|.\|}(\zeta_{4,5})\le \fil$ 
connecting $\gamma(t)$ with $\gamma(a)$ and $\gamma(b)$ with $\gamma(t)$.
Consider a lift $\overline{\gamma}$ of $\gamma$ to $V$ and a lift $\overline{\xi}_2$ of 
$\zeta_4\ast\gamma|_{[a,b]}\ast\zeta_5\ast \gamma|_{[t,a]}$ starting at
$\overline{\gamma}(t)$. Let $q$ be the terminal point of $\overline{\xi}_2$. Then we have
$$\overline{\gamma}(\tau)-q=[\gamma(\tau)-\gamma(b)]+\sum_{i=4}^5 [\zeta_i(b_i)-\zeta_i(a_i)].$$
By construction we have $\|\sum_{i=4}^5 [\zeta_i(b_i)-\zeta_i(a_i)]\|\le 2\fil$. Since 
$\dist(\gamma(\tau)-\gamma(b),\partial\T)\ge 3\fil$ we get 
$\overline{\gamma}(\tau)-q\in \T$. Choose a futurepointing curve $\zeta_6\colon [a_6,b_6]\to V/\Gamma$ with 
$\zeta_6(a_6)=\gamma(a)$ and $\zeta_6(b_6)-\zeta_6(a_6)=\overline{\gamma}(\tau)-q$.

Set $F_0:=F+\sup f \sqrt{C}K_0$. We claim that $\gamma'$ and $\gamma''$ are $F_0$-almost 
maximal. Indeed we have
\begin{align*}
L^g(\gamma')\ge L^g(\gamma|_{[a,t]})-L^g(\gamma|_{[b,\tau]})&\ge 
d(\overline{\gamma}(a),\overline{\gamma}(t))-L^g(\gamma|_{[b,\tau]})-F\\
&\ge d(\overline{\gamma}(a),\overline{\gamma}(t))-(F+\sup f \sqrt{C}K_0).
\end{align*}
Analogously we get $L^g(\gamma')\ge d(\overline{\gamma}(s),\overline{\gamma}(\tau))-(F+\sup f \sqrt{C}K_0)$.

Set $v:=[\gamma(b)-\gamma(a)]+[\zeta_1(b_1)-\zeta_1(a_1)]$ and $w:=\gamma(t)-\gamma(s)$.
Note that we have $\|v\|\le 2G_0+\fil$ and $\dist(v,\partial\T)\ge \e_0\e$. 
If $\|v+w\|\ge \frac{4(F_0+1)^2}{\inf f^2 c\;\e_0 \e}$ we get
\begin{align*}
\inf f|v+w|_1&\ge \inf f\sqrt{c\dist(v+w,\partial\T)\|v+w\|}\\
&\ge \inf f\sqrt{c\;\e_0 \e\|v+w\|}\ge 2(F_0+1).
\end{align*}
Consequently we have 
$$\sup f(|v|_1+|w|_1)> \frac{\inf f}{2} |v+w|_1.$$
Set $\lambda_1:=\frac{2\sup f}{\inf f}$. Recall the definition of $\mu:=\frac{c\;\eta}{4\lambda_1^2 C}$. If 
$\|w\|\ge \mu\; (2G_0+\fil)$ we get 
$$\dist(w,\partial \T)> \mu\dist(v,\partial \T)\ge \mu\; \e_0 \e$$
with lemma \ref{L2}. 

Set 
$$K:= \frac{4}{\eta}\max\left\{\mu(2G_0+\fil),\frac{4(F_0+1)^2}{\inf f^2 c\;\e_0 \e\;},G_0+K_0\right\}.$$ 
Let $[s,t]\subset I$ with $\|\gamma(t)-\gamma(s)\|\ge K$. We want to show that there exists $\delta>0$
such that $\gamma(t)-\gamma(s)\in \T_\delta$. If $\|\gamma(t)-\gamma(s)\|\ge 2K$ we can partition $[s,t]$ into 
mutually disjoint subintervals $[s_i,t_i]$ with $\|\gamma(t_i)-\gamma(s_i)\|\in [K,2K]$. If we have 
$\gamma(t_i)-\gamma(s_i)\in\T_\delta$ for all $i$ and some $\delta>0$, we get $\gamma(t)-\gamma(s)\in\T_\delta$, 
using corollary \ref{C0}. Therefore we can assume that $\|\gamma(t)-\gamma(s)\|\in [K,2K]$. 
If we have $[s,t]\setminus (a,\tau_0)= \emptyset$, we are done since we can then apply the above cut-and paste 
argument and obtain $\dist(\gamma(t)-\gamma(s),\partial\T)\ge \mu \e_0\e$. Then we have 
$$\gamma(t)-\gamma(s)\in \T_{\delta'}$$
for $\delta':=\frac{\mu\e_0\e}{2K}$. 

If we have $[s,t]\setminus (a,\tau_0)\neq \emptyset$, then $[s,a]$ or $[\tau,t]\neq \emptyset$. By the choice of $K$, 
we have
$$\max\{\|\gamma(a)-\gamma(s)\|,\|\gamma(t)-\gamma(\tau)\|\}
\ge \frac{1}{\eta}\max\left\{\mu(2G_0+\fil),\frac{4(F_0+1)^2}{\inf f^2 c\;\e_0 \e\;}\right\}.$$ 
Recall that we have $\gamma(b)-\gamma(a)\in \T_{\e_0}$ and $\gamma(\tau_0)-\gamma(b)\in \T_{\e'}$ 
for $\e':=\frac{3\fil}{K_0}$. Again with the above cut-and-paste argument we obtain 
\begin{equation}\label{E6}
\gamma(t)-\gamma(\tau)\text{ or }\gamma(a)-\gamma(s)\in \T_{\delta''}
\end{equation}
for $\delta'':=\min\{\delta', \e_0,\e'\}$.
Note the following fact. Let $v,w\in \T$ with $\|w\|\le \|v\|$ and $v\in \T_\e$. Then we have $v+w\in \T_{\e/2}$. 
Combining this with (\ref{E6}) we get 
$$\gamma(t)-\gamma(s)\in \T_{\delta},$$
for $\delta:=\frac{1}{2}\delta''$. This finishes the proof.
\end{proof}

\section{The Stable Time Separation}\label{S4}

\begin{proposition}\label{P10}
There exists a positively homogenous concave function $\mathfrak{l}\colon \T\to [0,\infty)$ such that:
\begin{enumerate}
\item For every $\e>0$ there exists a $K(\e)>0$ such that
$$|\mathfrak{l}(v)-d(x,x+v)|\le K(\e)$$
for all $v\in\T_\e$ and all $x\in V$.
\item $\inf f |v|_1\le \mathfrak{l}(v)\le \sup f |v|_1$ for all $v\in \T$.
\item $\mathfrak{l}(v+w)\ge \mathfrak{l}(v)+\mathfrak{l}(w)$ for all $v,w\in\T$.
\end{enumerate}
\end{proposition}

The following lemma is an adapted version of lemma 1 in \cite{bu}.
\begin{lemma}\label{L10-4}
Let $K<\infty$ and $f\colon [0,\infty)\to [0,\infty)$ be a $L$-Lipschitz continuous function such that
\begin{enumerate}
\item $f(2t)-2f(t)\le K$,
\item $zf(t)-f(zt)\le K$ for $z=2,3$
\end{enumerate}
and all $t\ge 0$. Then there exists an $a\in \mathbb{R}$, such that
$$|f(t)-at|\le 2K,$$
for all $t\ge 0$.
\end{lemma}

\begin{proof}
We have 
$$z^n f(t)\le f(z^n t)+K\sum_{k=1}^n z^k$$
for all $t\ge 0$, $z=2,3$ and integers $n\in\Z_{\ge 0}$. This implies 
$$\frac{f(t)}{t}\le \frac{f(z^n t)}{z^n t}+\frac{2K}{t}.$$
From $f(t)\le L t+f(0)$ we get the existence of $\limsup_{n\to \infty}\frac{f(z^n t)}{z^n t}=:a_z(t)$.
Choose for $\e>0$ an integer $r\in\Z_{\ge 0}$ such that $\frac{f(z^{r}t)}{z^{r}t}\ge a_z(t)-\e$ and 
$\frac{2K}{z^{r}t}\le\e$. Then we have
$$\frac{f(z^{n+r}t)}{z^{n+r}t}\ge \frac{f(z^{r}t)}{z^{r}t}-\frac{2K}{z^{r}t}\ge a_z(t)-2\e$$
for all $n\ge 0$. Therefore the sequence $\left\{\frac{f(z^n t)}{z^n t}\right\}_{n}$ converges to 
$a_z(t)$.

We claim that $a_z(t)$ is independent of $t\ge 0$ and $z=2,3$. We have 
\begin{align*}
\left|\frac{f(2^n t)}{2^n t}-\frac{f(3^m s)}{3^m s}\right|&\le \left|\frac{f(2^n t)}{2^n t}
-\frac{f(3^m s)}{2^n t}\right|+\left|\frac{f(3^m s)}{2^n t}-\frac{f(3^m s)}{3^m s}\right|\\
&\le L\frac{|2^n t-3^m s|}{2^n t}+\frac{|2^n t-3^m s|}{2^n t}\frac{f(3^m s)}{3^m s}
\end{align*}
for all $s,t\ge 0$. Since $\liminf_{m,n\to \infty} \frac{|2^n t-3^m s|}{2^n t} =0$ we get $a_2(t)=a_3(s)=:a$. 

Define $f(t)-at =:\delta (t)$. Then we get
$$|f(2^n t)-2^n(at+\delta(t))|\le \sum_{k=1}^n 2^k K\le 2^{n+1} K$$
and 
$$\left|\frac{f(2^n t)}{2^n t}-\left(a+\frac{\delta(t)}{t}\right)\right|<\frac{2K}{t}.$$
Passing to the limit $n\to \infty$, we obtain the lemma.
\end{proof}

\begin{lemma}\label{L10-3}
For all $\varepsilon >0$ there exists a $K'=K'(\varepsilon)<\infty$ such that
$$d(x,x+zv)\ge zd(x,x+v)-K'$$
for $z=2,3$ and all $x,v\in V$ with $v\in\mathfrak{T}_\varepsilon$.
\end{lemma}

\begin{proof}
Recall the definition of $\diam(\Gamma,\|.\|)$.
Choose $l\in \Gamma$ with $x+l\in B_{\diam(\Gamma,\|.\|)}(x+v)$. Using theorem \ref{T2a} we have
\begin{align*}
d(x,x+2v)&\ge d(x,x+v)+d(x+v,x+2v)\\
&\ge d(x,x+v)+d(x+l,x+l+v)-2L(\e)\|v-l\|\\
&=2d(x,x+v)-2L(\e)\diam(\Gamma,\|.\|).
\end{align*}
The other case follows analogously.
\end{proof}

\begin{lemma}[\cite{bu}, lemma 2]\label{L10--}
Let $\overline{\gamma}\colon [a,b]\to V$ be a continuous curve. Then there exist $k\le [m/2]$ and $k$-many 
mutually disjoint closed subintervals $[s_i,t_i]_{1\le i\le k}\subseteq [a,b]$ with
$$\sum_{i=1}^k [\overline{\gamma}(t_i)-\overline{\gamma}(s_i)]
=\frac{1}{2}[\overline{\gamma}(b)-\overline{\gamma}(a)].$$
\end{lemma}

\begin{lemma}\label{L10-}
For all $\varepsilon >0$ there exists a $K''=K''(\varepsilon)<\infty$ such that
$$2d(x,x+v)\ge d(x,x+2v)-K''$$
for all $x,v\in V$ with $v\in \mathfrak{T}_\varepsilon$.
\end{lemma}

\begin{proof}
The idea is to build an almost maximal curve between $x$ and a point near $x+v$ from 
pieces of a maximal curve between $x$ and $x+2v$. By increasing $K''$, we can assume 
$\|v\| \ge (2m+1)\fil$. Let $\overline{\gamma}\colon [a,b]\to V$ be a maximal curve from $x$ to $x+2v$. 
Choose intervals $[s_i,t_i]\subset [a,b]$ as in lemma \ref{L10--}. Next choose points 
$\overline{\gamma}(s_i)^\ast$ and $\overline{\gamma}(t_i)^\ast$ in the $\Gamma$-orbit of $x$ such that 
$\overline{\gamma}(s_i)^\ast\in (\overline{\gamma}(s_i)+\T)\cap B_{\fil}(\overline{\gamma}(s_i))$ and 
$\overline{\gamma}(t_i)^\ast\in (\overline{\gamma}(t_i)+\T)\cap B_{\fil}(\overline{\gamma}(t_i))$. 
Applying the transformations in $\Gamma$ that map $\overline{\gamma}(s_i)^\ast$ to 
$\overline{\gamma}(t_{i-1})^\ast$, resp. $\overline{\gamma}(s_1)^\ast$ to $x$, yields:
$$d(x,x+v+\sum [\overline{\gamma}(s_i)^\ast -\overline{\gamma}(s_i)]+[\overline{\gamma}(t_i)^\ast 
-\overline{\gamma}(t_i)])\ge \sum d(\overline{\gamma}(s_i),\overline{\gamma}(t_i))$$
We have $v+\sum [\overline{\gamma}(s_i)^\ast -\overline{\gamma}(s_i)]+[\overline{\gamma}(t_i)^\ast 
-\overline{\gamma}(t_i)]\in \T_{\frac{\e}{2}}$ since we have assumed $\|v\| \ge (2m+1)\fil$.
With the Lipschitz continuity of the time separation on $\T_{\frac{\e}{2}}$ we get
\begin{align*}
d(x,x&+v+\sum [\overline{\gamma}(s_i)^\ast -\overline{\gamma}(s_i)]+[\overline{\gamma}(t_i)^\ast 
-\overline{\gamma}(t_i)])\\
&\le d(x,x+v)+L(\e/2)\sum \dist(\overline{\gamma}(s_i),\overline{\gamma}(s_i)^\ast)
+\dist(\overline{\gamma}(t_i),\overline{\gamma}(t_i)^\ast)\\
&\le d(x,x+v)+mL(\e/2)\fil.
\end{align*}
Since we can repeat the argument with $\overline{\gamma}|_{[s_i,t_i]}$ replaced by 
$\overline{\gamma}|_{[t_{i-1},s_i]}$ we get after summing of the results:
\begin{align*}
d(x,x+2v)&=\sum  d(\overline{\gamma}(s_i),\overline{\gamma}(t_i))+d(\overline{\gamma}(t_{i-1}),\overline{\gamma}(s_i)) \\
&\le 2d(x,x+v)+2mL(\e/2)\fil
\end{align*}
\end{proof}

\begin{proof}[Proof of proposition \ref{P10}]
Lemma \ref{L10-} and \ref{L10-3} ensure that lemma \ref{L10-4} can be applied to $f_x(t):=d(x,x+tv)$. Then there 
exists an $a_x(v)$ with $|d(x,x+tv)-a_x(v)t|\le 2\max\{K',K''\}=:K$. In fact $a_x(v)$ does not depend on $x$ since 
for $x'\in V$ we have
$$\left|\frac{d(x',x'+nv)}{n}-a_x(v)\right| \le \frac{K+2L(\e)\fil}{n}.$$
This shows the independence of $a_x(v)$ of $x$ as well as the uniform convergence on compact subsets of $\T_\e$ 
of $\frac{1}{n}d(x,x+nv)$ to $a_x(v)$. Set $\mathfrak{l}(v):=a_x(v)$. 
The estimate (1) then follows from the definition. Property (2) follows directly from the estimate:
$$\inf f |v|_1\le d(x,x+v)\le \sup f |v|_1$$
The inverse triangle inequality follows readily from the inverse triangle inequality for the time separation. 
For the positive homogeneity note that by the uniform convergence on compact subsets of $\T_\e$ it suffices to consider 
rational factors $\eta=\frac{p}{q}$. Then by considering subsequences we get
\begin{align*}
\mathfrak{l}(\eta v)&=\lim \frac{1}{n}d(x,x+\frac{p}{q}nv)=\lim \frac{1}{qn}d(x,x+pnv)\\
&=\lim \frac{p}{qn}d(x,x+nv)=\frac{p}{q}\lim \frac{1}{n}d(x,x+nv)=\eta \mathfrak{l}(v).
\end{align*}
\end{proof}

\section{The Rotation Vector}\label{S5}

Define the rotation vector of a future pointing curve $\gamma\colon [a,b]\to V/\Gamma$: 
$$\rho(\gamma):=\frac{1}{\mathfrak{l}(\gamma(b)-\gamma(a))}[\gamma(b)-\gamma(a)]$$

\begin{theorem}\label{T10}
Let $\e>0$ and $\gamma\colon \R\to V/\Gamma$ be a homologically maximizing geodesic with $\dot{\gamma}(t_0)\in  \T_\e$ 
for some $t_0\in \R$. Then there exists a support function $\alpha$ of $\mathfrak{l}$ such 
that for all neighborhoods $U$ of $\alpha^{-1}(1)\cap \mathfrak{l}^{-1}(1)$ there exists a $K=K(\e,U)>0$ such that for 
all $s<t\in \R$ with $\|\gamma(t)-\gamma(s)\|\ge K$, we have 
$$\rho(\gamma|_{[s,t]})\in U.$$
\end{theorem}

\begin{lemma}\label{L10}
Let $\e,\delta>0$, $F,G<\infty$ and $n\in \N$ be given. Then there exists a $K=K(\e,\delta, F,G,n)<\infty$ such that 
for all $k\le n$, all $T\ge K$ and all $F$-almost maximal $(G,\e)$-timelike curves $\gamma\colon \R\to V$, the 
following holds.

Given $k$ many intervals $[t_i,t_i+\sigma_i]$ with disjoint interiors and $L^g(\gamma|_{[t_i,t_i+\sigma_i]})$ $=T$ 
we have
$$\mathfrak{l}(\sum \rho(\gamma|_{[t_i,t_i+\sigma_i]}))\le k+1+\delta.$$
\end{lemma}

\begin{proof}
Assume that the arcs $\gamma|_{[t_i,t_i+\sigma_i]}$ are indexed in increasing order. W.l.o.g. we can suppose that 
$\|\gamma(t_i+\sigma_i)-\gamma(t_{i+1})\|\le\fil$. If this is not the case we can repeat the "cut-and-paste" operation
from the proof of proposition \ref{P1}. Choose $\delta >0$ according to theorem \ref{T2}. Notice that in this case for 
any interval $[s_\infty, t_\infty]$ with $\|\gamma(t_\infty)-\gamma (s_\infty)\|= (2n-1)\frac{\fil}{\delta}$, we have
$\dist(\gamma(t_\infty)-\gamma(s_\infty),\partial\T)\ge (2n-1)\fil$. Choose one such interval disjoint from 
$[t_i,t_i+\sigma_i]$. The new curve, resulting from the cut-and-paste operation, will be 
$L^g(\gamma|_{[s_\infty,t_\infty]})-F$-almost maximal. Define $A:=\max L^g(\gamma|_{[t_i+\sigma_i,t_{i+1}]})$.

Now let $\T'$ be a proper sub-cone of $\T$, $v_i\in \T'$ ($0\le i\le k$) with $\mathfrak{l}(v_i)=1$ and 
$\mathfrak{l}(\sum v_i)= k+1+r$. Further let $\lambda_i> 0$ and $\lambda \le \min \lambda_i$. Then
$$\mathfrak{l}(\sum \lambda_i v_i)\ge \lambda \mathfrak{l}(\sum v_i)+\sum \mathfrak{l}((\lambda_i-\lambda)v_i)\ge 
\lambda (k+1+r)$$
and for $\lambda_i=\mathfrak{l}(\gamma(t_i+\sigma_i)-\gamma(t_i))$
$$\mathfrak{l}(\sum \gamma(t_i+\sigma_i)-\gamma(t_i))\ge \lambda \mathfrak{l}(\sum \rho(\gamma|_{[t_i,t_i+\sigma_i]})).$$

By proposition \ref{P10} there exists a $K<\infty$, depending only on $F,G$ and $\e$, such that 
$|\mathfrak{l}(\gamma(t)-\gamma(s)) -L^g(\gamma|_{[s,t]})|\le K$ for all $s, t \in \R$.
For $T> K$ set $\lambda := T-K$. Then we conclude
\begin{align*}
(k+1+r)(T-K)&\le \sum L^g(\gamma|_{[t_i,t_i+\sigma_i]})\\
&\le L^g(\gamma|_{[t_1,t_k+\sigma_k]})= (k+1)T +k(A+K). 
\end{align*}
Solving for $r$ shows
$$r\le \frac{1}{T-K}((2k+1)K+kA).$$
Increasing $T$ sufficiently we conclude the assertion.
\end{proof}

\begin{lemma}\label{L11}
Let $\{W_n\}_{n\in\N}$ be a sequence of subsets of $\mathfrak{l}^{-1}(1)$  such that $W_{n+1}\subset cone(W_n)$. Assume 
further that there exists a sequence $\delta_n \downarrow 0$ such that for any $k$-tuple of pairwise
different $v_i\in W_n$, $\mathfrak{l}(\sum v_i)\le k+1+\delta_n$ holds. Then there exists a supporting hyperplane $E$ of 
$\mathfrak{l}^{-1}(1)$ such that for any neighborhood $U$ of $E\cap \mathfrak{l}^{-1}(1)$ the intersection 
$cone(W_n)\cap \mathfrak{l}^{-1}(1)$ is eventually contained in $U$. 
\end{lemma}

\begin{proof}
Let $\{v_0,\ldots ,v_k\}\subset W_n$ and $t_0,\ldots ,t_k\ge 0$ with $\sum t_i =1$. Since $\mathfrak{l}(v_i)=1$ we 
conclude
\begin{align*}
k+1+\delta_n\ge \mathfrak{l}(\sum v_i)&\ge \mathfrak{l}(\sum t_i v_i)+\sum \mathfrak{l}((1-t_i)v_i)\\
&=\mathfrak{l}(\sum t_i v_i)+\sum (1-t_i)=\mathfrak{l}(\sum t_i v_i)+k.
\end{align*}
Therefore $\mathfrak{l}(\sum t_i v_i)\le 1+\delta_n$. By Caratheodory's theorem the convex hull of $W_n$ is the union 
of all simplices with vertices in $W_n$. Consequently the closure of the convex hull of $W_n$ is disjoint from 
$\mathfrak{l}^{-1}([1+2\delta_n,\infty))$. By the inverse triangle inequality, proposition \ref{P10} (ii) and  the 
assumption $W_{n+1}\subset cone(W_n)$, the sets $W_n$ are uniformly bounded. Then there exists an affine hyperplane 
$E_n$ separating the convex hull of $W_n$ from $\mathfrak{l}^{-1}([1+2\delta_n,\infty))$. Now consider a limit 
hyperplane $E$ of $\{E_n\}_{n\in\N}$. Since $\lim \delta_n =0$, $E$ is a supporting hyperplane of $\mathfrak{l}^{-1}(1)$.
The assertion now follows easily. 
\end{proof}

\begin{proof}[Proof of theorem \ref{T10}]
It suffices to prove the assertion for para\-meter intervals of the form $[i2^n,(i+1)2^n]$ for $i\in \Z$ and $n\in\N$.  
Set $W_n:=\{\rho(\gamma|_{[i2^n,(i+1)2^n]})|\, i\in \Z\}$.
By lemma \ref{L10} there exists a sequence $\delta_n\downarrow 0$ such that for all $\{v_0,\ldots ,v_k\}\subset W_n$ 
($k\le m$) holds $\mathfrak{l}(v_0+\ldots v_k)\le k+1+\delta_n$. Since any vector $\gamma((i+1)2^{n+1})-\gamma(i2^{n+1})$
is the sum of two vectors of the form $\gamma((j+1)2^{n})-\gamma(j2^{n})$, the convex cone over $W_n$ contains 
$W_{n+1}$. This establishes the assumptions of lemma \ref{L11} and therefore the assertion.
\end{proof}

Recall the definition of the dual cone $\mathfrak{T}^\ast:=\{\alpha \in V^\ast |\, \alpha|_\mathfrak{T} \ge 0\}$ of $\T$.
We call a function $h\colon V\to \mathbb{R}$ $\alpha$-equivariant, if  
$$h(x+k)=h(x)+\alpha(k)$$
for all $x\in V$ and $k\in \Gamma$.

\begin{remark}
For $\alpha\in V^\ast$ there exists an $\alpha$-equivariant time function if and only if $\alpha \in (\T^\ast)^\circ$.
Furthermore, the existence of an $\alpha$-equivariant time function is equivalent to the existence of a smooth 
$\alpha$-equivariant temporal function (a $C^1$-function is a temporal function if the Lorentzian gradient is always 
timelike and past pointing).
\end{remark}

Let $\alpha\in V^\ast$ and $h\colon V\to \R$ be an $\alpha$-equivariant $C^1$-function. Then the $1$-form $dh$ descends
to a $1$-form $\omega_h$ on $V/\Gamma$.

\begin{definition}
Let $\alpha\in (\T^\ast)^\circ$ and $\tau \colon V\to \R$ be an $\alpha$-equivariant temporal function.
\begin{enumerate}
\item Define for $\sigma \in\R$:
$$\mathfrak{h}_\tau(\sigma):=\sup\{L^{g}(\gamma)|\;\gamma\text{ future pointing, }\int_\gamma \omega_\tau=\sigma\}$$ 
\item A homologically maximizing curve $\gamma\colon I\to V/\Gamma$ is said to be $\alpha$-almost 
maximal if there exists a constant $F<\infty$ such that
$$L^{g}(\gamma|_{[s,t]})\ge \mathfrak{h}_\tau\left(\int_{\gamma|_{[s,t]}}\omega_\tau\right)-F$$
for all $s<t\in I$.
\end{enumerate}
\end{definition}

\begin{remark}

(i) Note that the definition of $\alpha$-almost maximality does not depend on the choice of $\alpha$-equivariant temporal 
function. A different choice of $\alpha$-equivariant time function only yields a different constant $F'$.

(ii) In Riemannian geometry the notion of $\alpha$-almost minimality (analogue of $\alpha$-alomost maximality) is now 
replaced by the notions of calibrations and calibrated curves (compare \cite{ba2}). The Lorentzian versions of
calibrations and calibrated curves will be introduced in \cite{diss}.
\end{remark}

Define the dual stable time separation 
$$\mathfrak{l}^\ast\colon \T^\ast\to \R,\; \alpha\mapsto \min\{\alpha(v)|\;\mathfrak{l}(v)=1\}.$$

\begin{theorem}\label{T20}
\begin{enumerate}
\item For every $\alpha\in (\T^\ast)^\circ$ there exists an $\alpha$-almost maximal timelike geodesic 
$\gamma\colon \R\to V/\Gamma$.
\item Let $\alpha\in \T^\ast$ with $\mathfrak{l}^\ast(\alpha)=1$. Then for every neighborhood $U$ of 
$\alpha^{-1}(1)\cap \mathfrak{l}^{-1}(1)$ there exists a $K=K(\alpha,U)<\infty$ such that 
$$\rho(\gamma|_{[s,t]})\in U$$
for all $\alpha$-almost maximal future pointing curves $\gamma\colon\R\to V/\Gamma$ and every $s<t\in \R$ with 
$\|\gamma(t)-\gamma(t)\|\ge K$.
\end{enumerate}
\end{theorem}

\begin{corollary}
$(V/\Gamma,g)$ contains infinitely many geometrically distinct homologically maximizing timelike geodesics 
$\gamma\colon \R\to V/\Gamma$ with the additional property that the limit
$$\lim_{t\to\infty}\rho(\gamma|_{[s,s+t]})=:v$$
exists uniformly in $s\in \R$. The $v$ are exposed points of $\mathfrak{l}^{-1}(1)$.
\end{corollary}

\begin{lemma}\label{L20}
Let $\alpha\in(\mathfrak{T}^\ast)^\circ$ and $\tau\colon V\to \R$ an $\alpha$-equivariant time function. 

(i) Then there exists an $\varepsilon=\varepsilon(\alpha)>0$, such that 
$$y-x\in \mathfrak{T}_\varepsilon$$
for all $x,y\in V$ with $y-x\in \mathfrak{T}$, $\tau(y)-\tau(x)\ge 2$ and $2d(x,y)\ge 
\mathfrak{h}_\tau(\tau(y)-\tau(x))$.

(ii) There exists a constant $I=I(\alpha)<\infty$ such that
$$\mathfrak{h}_\tau(\sigma_1)+\mathfrak{h}_\tau(\sigma_2)-I\le \mathfrak{h}_\tau(\sigma_1+\sigma_2)\le 
\mathfrak{h}_\tau(\sigma_1)+\mathfrak{h}_\tau(\sigma_2)$$
for all $\sigma_{1},\sigma_{2}\ge 0$.

(iii) We have 
$$\lim_{\sigma\to \infty}\frac{\sigma}{\mathfrak{h}_\tau(\sigma)}=\mathfrak{l}^\ast(\alpha).$$
\end{lemma}

\begin{proof}
(i) Clear from the fact that $\tau$ is an $\alpha$-equvariant time function and theorem \ref{T2}.

(ii) The right side of the inequality is clear from the observation that any curve $\gamma\colon [a,b]\to V/\Gamma$ with 
$\int_{\gamma} \omega_\tau=\sigma_1 +\sigma_2$ is the concatenation of two curves $\gamma_{i}\colon [a_{i},b_{i}]\to 
V/\Gamma$ $(i=1,2)$ with $\int_{\gamma_i} \omega_\tau=\sigma_{i}$.

For the other inequality we reverse this procedure. Let $\gamma_{i}\colon [a_{i},b_{i}]\to V/\Gamma$ be future pointing 
curves with $\int_{\gamma_i} \omega_\tau =\sigma_i$ $(i=1,2)$. Choose a future pointing curve $\zeta$ from $\gamma_1(b_1)$
to $\gamma_2(a_2)$ of $\|.\|$-arclength bounded by $\fil$. Then for sufficiently large $\sigma_2$ (the other cases can 
be absorbed into the constant $I$) there exists a $b'_2\in [a_2,b_2]$ with $\int_{\gamma_1\ast \zeta\ast 
\gamma_2|_{[a_2,b_2']}} \omega_\tau=\sigma_1+\sigma_2$. Using (i), the distance $\|\gamma_2(b_2)-\gamma_2(b'_2)\|$ 
is uniformly bounded by some $I'(\alpha)<\infty$. The claim now follows by noting that $d(x,x+v)\le \sup f |v|_1$. 

(iii) Choose $v_\alpha \in \mathfrak{l}^{-1}(1)$ with $\alpha(v_\alpha)=1$. Then for all $x\in V$
$$\mathfrak{l}^\ast(\alpha)=\lim_{n\to\infty}\frac{n}{d(x,x+nv_\alpha)}\mathfrak{l}^\ast(\alpha)\ge \limsup_{n\to\infty}
\frac{n\mathfrak{l}^\ast(\alpha)}{\mathfrak{h}_\tau(\tau(x+nv_\alpha)-\tau(x))}.$$
The difference $\tau(x+nv_\alpha)-\tau(x)-n\alpha(v_\alpha)$ is uniformly bounded. 
The continuity of $\mathfrak{h}_\tau$ and part (ii) then imply
$$\limsup_{\sigma\to\infty}\frac{\sigma}{\mathfrak{h}_\tau(\sigma)}\le \mathfrak{l}^\ast(\alpha).$$
It remains to prove the opposite inequality
$$\liminf_{\sigma\to\infty}\frac{\sigma}{\mathfrak{h}_\tau(\sigma)}\ge \mathfrak{l}^\ast(\alpha).$$
By proposition \ref{P10}, there exists a $K=K(\alpha)<\infty$ such that $\mathfrak{l}(\sigma v')\ge 
d(x,x+\sigma v')-K$ for all $\sigma\ge 0$ and all $v'\in \mathfrak{l}^{-1}(1)\cap \alpha^{-1}(1)$. But then we have
$$\sigma \ge \mathfrak{l}^\ast(\alpha)\mathfrak{l}(\sigma v')\ge \mathfrak{l}^\ast(\alpha)
\mathfrak{h}_\tau(\tau(x+\sigma v')-\tau(x)).$$
The claim now follows by noting again that the difference $\tau(x+nv')-\tau(x)-n\alpha(v')$ is uniformly bounded. 
\end{proof}

\begin{proof}[Proof of theorem \ref{T20}]
(i) Let $\tau\colon V\to \R$ be an $\alpha$-equivariant temporal function. Consider a sequence of homologically 
maximizing timelike pregeodesics $\gamma_n\colon [-T_n,T_n]\to V/\Gamma$ parameterized by $\|.\|$-arclength such that 
$\int_{\gamma_n}\omega_\tau=2n$ and $L^g(\gamma_n)=\mathfrak{h}_\tau(2n)$. The sequence $T_n$ is obviously 
unbounded. The limit curve lemma implies that $\{\gamma_n\}$ contains a converging subsequence. 
Denote the limit curve by $\gamma\colon \R\to V/\Gamma$. 
The $\alpha$-maximality of the $\gamma_n$ imply that the $\gamma_n$ are uniformly timelike. $\gamma$ is therefore 
timelike as well. The homological maximality of $\gamma$ follows from the homological maximality of the $\gamma_n$.
$\gamma$ is an $\alpha$-almost maximal curve. Note that by lemma \ref{L20} (ii) for any interval $[a,b]\subset 
[-T_n,T_n]$ we have
\begin{align*}
L^g(\gamma_n)&=L^g(\gamma_n|_{[-T_n,a]})+L^g(\gamma_n|_{[a,b]})+L^g(\gamma_n|_{[b,T_n]})= \mathfrak{h}_\tau(2n)\\
&\ge \mathfrak{h}_\tau\left(\int_{\gamma_n|_{[-T_n,a]}}\omega_\tau\right)
+\mathfrak{h}_\tau\left(\int_{\gamma_n|_{[a,b]}}\omega_\tau\right)
+\mathfrak{h}_\tau\left(\int_{\gamma_n|_{[b,T_n]}}\omega_\tau\right)-2F.
\end{align*}
This implies already $L^g(\gamma|_{[a,b]})\ge \mathfrak{h}_\tau\left(\int_{\gamma_n|_{[a,b]}}\omega_\tau\right) -2F$. 

(ii) The first step is to note that for all $\alpha$-almost maximal curves $\gamma$ and all $\alpha$-equivariant 
temporal functions $\tau$, the ratio 
\begin{align}\label{E20}
\frac{\int_{\gamma|_{[s,t]}}\omega_\tau}{L^g(\gamma|_{[s,t]})}\to 1
\end{align}
for $\|\gamma(t)-\gamma(s)\|\to \infty$. Lemma \ref{L20} (iii) implies $\lim_{\sigma\to \infty}
\frac{\sigma}{\mathfrak{h}_\tau(\sigma)}=1$. Since $\gamma$ is $\alpha$-almost maximal, we conclude
$$\lim_{\|\gamma(t)-\gamma(s)\|\to \infty}
\frac{\mathfrak{h}_\tau(\int_{\gamma|_{[s,t]}}\omega_\tau)}{L^g(\gamma|_{[s,t]})}=1.$$
This implies (\ref{E20}).

By lemma \ref{L20} (i), there exists an $\e(\alpha)>0$ such that for $\|\gamma(t)-\gamma(s)\|$ sufficiently large,
$\gamma(t)-\gamma(s)\in \T_{\e(\alpha)}$. But then 
$$L^g(\gamma|_{[s,t]})-K(\e)\le \mathfrak{l}(\gamma(t)-\gamma(s))\le \alpha(\gamma(t)-\gamma(s)).$$
Now let $U$ be a neighborhood of $\alpha^{-1}(1)\cap \mathfrak{l}^{-1}(1)$ and $\delta=\delta(U)>0$ such that for all 
$h\in \mathfrak{l}^{-1}(1)\setminus U$, $\alpha(h)\ge 1+\delta$. If $\rho(\gamma(t)-\gamma(s))\notin U$, we get
$$\alpha(\gamma(t)-\gamma(s))\ge (1+\delta)\mathfrak{l}(\gamma(t)-\gamma(s))\ge (1+\delta)(L^g(\gamma|_{[s,t]})-K(\e)).$$
If $\gamma(t)-\gamma(s)$ is not bounded from above, this contradicts the above conclusion:
$\frac{\int_{\gamma|_{[s,t]}}\omega_\tau}{L^g(\gamma|_{[s,t]})}\to 1$
\end{proof}

\bibliographystyle{amsalpha}

\end{document}